\documentclass[reqno]{amsart}
\usepackage{amssymb, amsmath, amsthm, amscd, mathrsfs}
\usepackage{color}
\usepackage{paralist}
\usepackage{cite}
\usepackage{bigints}
\usepackage{dsfont}
\usepackage{empheq}
\usepackage{amssymb}
\usepackage{cases}
\usepackage{enumitem}
\allowdisplaybreaks

\usepackage{subfig}

\usepackage{caption}
\usepackage{booktabs}

\usepackage{float}
\usepackage[ruled,vlined,linesnumbered,noresetcount]{algorithm2e}

\usepackage[pagebackref]{hyperref}
\renewcommand*{\backref}[1]{}\renewcommand*{\backrefalt}[4]{\ifcase #1 (\tt not cited)\or (\tt cited on page~#2)\else (\tt cited on pages~#2)\fi}

\theoremstyle{plain}
\newtheorem{theorem}{Theorem}[section]

\newtheorem{lemma}[theorem]{Lemma}
\newtheorem{proposition}{Proposition}

\theoremstyle{definition}
\newtheorem{definition}[theorem]{Definition}
\newtheorem{remark}{Remark}
\newtheorem{hypothesis}{Hypothesis}
\numberwithin{equation}{section}

\usepackage{url}

\usepackage{lipsum}

\author[Hind El Baggari]{ Hind El Baggari \textsuperscript{1}}
\address{\textsuperscript{1}Cadi Ayyad University, Faculty of Sciences Semlalia, Laboratory of Mathematics, Modeling and Automatic Systems, Marrakech, Morocco.}
\email{hindbaggari@gmail.com}

\author[Ilham Ouelddris]{Ilham Ouelddris \textsuperscript{1}}
\address{\textsuperscript{1}Cadi Ayyad University, Faculty of Sciences Semlalia, Laboratory of Mathematics, Modeling and Automatic Systems, Marrakech, Morocco.}
\email{ouelddris.ilham@gmail.com}

\makeatletter 
\@namedef{subjclassname@2020}{%
	\textup{2020} Mathematics Subject Classification}
\makeatother


\title{ Logarithmic Convexity and impulse Approximate Controllability for Degenerate Parabolic Equations with Robin Boundary Conditions}

\begin{document}

\begin{abstract}
    In this work, we investigate the approximate controllability of a class of one-dimensional degenerate parabolic equations with Robin boundary conditions. The degeneracy occurs at one endpoint of the spatial domain, and we apply an impulsive control in a small region at a fixed moment. Our main result establishes an observability inequality for the adjoint system, from which we deduce approximate controllability at final time . The proof relies on a logarithmic convexity argument, developed through a Carleman commutator approach.
\end{abstract}

\maketitle

\section{Introduction}

Impulsive systems are a class of dynamical systems that experience sudden changes at specific moments in time. They appear in many areas like engineering, biology, population studies, and economics  see e.g., \cite{Agarwal2017, Terzieva2018}.

In this paper, we study the approximate controllability of a degenerate parabolic equation with Robin boundary conditions and impulsive control acting at a fixed time $\tau\in (0,T)$. The equation is given by
  \begin{empheq}[left = \empheqlbrace]{alignat=2} \label{1.1}
		\begin{aligned}
		&\partial_t y(x,t) - (x^{\alpha} y_x)_x (x,t) = 0, && \quad\text { in } (0,1) \times(0, T)\backslash\{\tau\} ,\\
            &y(x,\tau)=y(x,\tau^{-})+1_\omega(x) f(x) && \quad x\in (0,1),\\
            &\beta_0 y (0,t)- \beta_1 (x^{\alpha} y_x)(0,t)=0, && \quad\text { in } (0, T) ,\\ 
            &y (1,t)=0, \\ 
			&y(x,0)= y_0(x),   && \quad \text{ on } (0,1).
		\end{aligned}
	\end{empheq}
Here, $\alpha\in \left(0,1\right)$ is the degeneracy rate, $\omega$ is a nonempty open subset of $(0,1)$, $\beta_0, \beta_1 \in \mathbb{R}$ and $y_0\in L^2(0,1)$ denotes the initial data.
First, let us present the following assumption
\begin{hypothesis}\label{hypo1}
    In what follows, we will assume that :
    \begin{enumerate}
        \item The control is supported in the right neighborhood of $x=0$, i.e.,
		\begin{equation}\label{cond}
			\omega:=(0,\kappa), \qquad \text{for some constant} \;\kappa\in (0,1).
		\end{equation}
        \item The paramets $\beta_0$ and $\beta_1$ satisify $\beta_0 \beta_1 \geq 0$ and $\beta_1 \neq 0$.
    \end{enumerate}
\end{hypothesis}

Robin boundary conditions can be described as a combination of Dirichlet and Neumann conditions, and they are frequently referred to as convective boundary conditions in heat transfer applications \cite{Hahn2012}. These conditions illustrate cases where the heat flow at the boundary is proportional to the temperature difference between the system and its surrounding environment. For a more comprehensive physical interpretation of Robin boundary conditions, we refer to the studies conducted by G.R. Goldstein \cite{Goldstein2006}.

Extensive studies has been conducted on the classical heat equation featuring linear Robin boundary conditions, with significant contributions from \cite{Daners2000} and \cite{Warma2006}. Furthermore, degenerate parabolic equations, which are present in several real-world situations, have received substantial academic investigation. A pertinent example is the Budyko-Sellers model in climatology, which examines tthe effects of continental and oceanic ice on climate change, thereby demonstrating the application of these equations \cite{Diaz2006, North2006}.

The analysis of degenerate parabolic problems can be traced back to significant foundational works in the area. Foundational results are found in \cite{Alabau-Boussouira2006} and \cite{Cannarsa2005}. Furthermore, \cite{Favini2001}, \cite{Taira1992}, and \cite{Taira1995} have delved into various boundary conditions for degenerate parabolic systems, utilizing a functional analytic approach to develop Feller semigroups associated with degenerate elliptic operators that incorporate Wentzell boundary conditions. Additionally, in \cite{Goldstein1991}, J.A. Goldstein and C.Y. Lin explore degenerate operators with a variety of boundary conditions, such as Dirichlet, Neumann, periodic, and nonlinear Robin types.

Degenerate elliptic operators also appear in models of gene frequency in population genetics, such as the Wright-Fischer model discussed in \cite{Shimakura1992}. More recently, \cite{NO'A} applied Robin boundary conditions to model heat transfer and established the equivalence between null-controllability and observability inequalities.

Impulsive dynamical systems have received considerable attention, especially concerning their existence, stability, and controllability characteristics. For comprehensive insights on these subjects, we direct readers to \cite{Jose2020, Lalvay2022}. Nonetheless, the body of literature focusing specifically on impulsive controllability is still somewhat sparse. Significant works in this area include \cite{BenAissa2021, Miller2003, Tao2001}, which utilize essential concepts such as logarithmic convexity and spectral inequalities.

The research presented in \cite{Phung2018} represents a significant development, as the authors introduced a Carleman commutator approach based on logarithmic convexity to derive observability estimates for parabolic equations governed by homogeneous Dirichlet boundary conditions. This technique was later modified for the Neumann context in \cite{BuffePhung2022}. Additional advancements regarding dynamic boundary conditions are discussed in \cite{Hind2023, Chorfi2022}.

In the context of degenerate parabolic equations with impulsive control, the research conducted in \cite{BuffePhung2018} demonstrated a Lebeau–Robbiano-type spectral inequality for a one-dimensional degenerate elliptic operator, which was subsequently utilized to achieve impulse control and finite-time stabilization. Additionally, the authors of \cite{Maarouf2023} examined the null approximate controllability of degenerate singular parabolic equations under the influence of impulsive actions.

In the present paper, we contribute to this growing body of research by studying the approximate controllability of a degenerate parabolic equation subject to Robin boundary conditions at the degeneracy point and a Dirichlet condition at the non-degenerate end. Our approach is based on establishing a logarithmic convexity property through a Carleman commutator method, following the ideas developed in \cite{Phung2018}.

More precisely, our main concern is to prove that the impulsive degenerate system with drift \eqref{1.1} is null approximately controllable according to the following definition.
	\begin{definition}[see Definition 1.2 of \cite{qin}]
		The system \eqref{1.1} is said to be null approximately impulse controllable at time $T>0$ if for any $\epsilon >0$ and $y_0 \in L^2(0,1)$, there exists a control $h(\cdot,\tau) \in L^2(\omega)$ such that the associated state $y$ at final time satisfies
		\begin{equation}\label{eqT5}
			\parallel y(\cdot, T) \parallel \leq \epsilon \parallel y_0 \parallel.
		\end{equation}
	\end{definition}
	When this null approximate impulse controllability property holds true, we have that for any given $\epsilon>0$ and $y_0 \in L^2(0,1)$,
	\begin{equation*}
		\mathcal{R}_{T,\epsilon,y_0} := \left\lbrace h(\cdot,\tau) \in L^2(\omega): \mathit{the \; solution} \; y \; \mathit{of \; system} \; \eqref{1.1} \; \mathit{satisfies} 	\parallel y(\cdot, T) \parallel \leq \epsilon \parallel y_0 \parallel \right\rbrace
	\end{equation*}
	is a non-empty set. We can then define the cost associated with null approximate impulse control as follows.
	\begin{definition}
		The quantity 
		\begin{equation*}
			\mathcal{P}(T,\epsilon) := \sup_{\parallel y_0 \parallel =1} \,\inf_{h(\cdot,\tau) \in \mathcal{R}_{T,\epsilon,y_0}} \parallel h(\cdot,\tau) \parallel_{L^2(\omega)}
		\end{equation*}
		is called the cost of null approximate impulse control at time $T$.
	\end{definition}
Next, we present our first main result, which is an observability estimate at a single time instant. The proof can be found in Section \ref{section3}.
\begin{lemma}\label{lem1}
		Assume that Hypothesis \ref{hypo1} holds true. Let us consider $\omega$ a sub-interval of $(0,1)$ that satisfies the geometric condition \eqref{cond}. Then, there exists a positive constant $\mathcal{C}$ and $\rho \in (0,1)$ such that the following observation estimate is satisfied 
		\begin{equation}\label{ob}
			\parallel u(\cdot,T) \parallel\leq \left(  e^{\mathcal{C}\left(1+\frac{1}{T-\tau}\right)} \parallel u(.,T) \parallel_{L^2(\omega)} \right)^{\rho} \parallel u(\cdot,0) \parallel^{1-\rho},
		\end{equation}
		where $u$ is the solution of the following non-impulsive system
		\begin{empheq}[left = \empheqlbrace]{alignat=2} \label{1.2}
			\begin{aligned}
				&\partial_{t} u- (x^{\alpha} u_x)_x=0 , && \quad\text { in } Q:=(0,1) \times (0,T),\\
                 &\beta_0 u (0,t)- \beta_1 (x^{\alpha} u_x)(0,t)=0, && \quad\text { in } (0, T) ,\\ 
				& u(1,t)= 0, \\
				& u(x,0)= u_{0}(x),  && \quad \text{ on } (0,1).
			\end{aligned}
		\end{empheq}
	\end{lemma}

The structure of the paper is as follows. Section \ref{section2} presents a review of the results concerning the well-posedness of the system. In Section \ref{section3}, we describe the approach taken to obtain the observation estimate at a specific time point for the system \eqref{1.1}. Section \ref{section4} focuses on the approximate controllability of the impulse of the system. Lastly, Section \ref{section5} concludes the paper.

\section{Well-Posedness of the system}\label{section2}
In this section, we recall some results that will be useful in the sequel. Now, let us define the following weighted Sobolev spaces:
\[H^1_{\alpha}(0,1):= \{u \in L^2(0,1): u \;\text{is absolutely continuous in}\; \left[0,1\right], x^{\frac{\alpha}{2}} u_x\in L^2(0,1) \;\text{and}\; u(1)=0\},\]
and
\begin{align*}
    H^2_{\alpha}(0,1)
    =&\{u\in H^1_{\alpha}(0,1): x^{\alpha} u_x \in H^1(0,1)\},
\end{align*}
with the norms
\[\|u\|^2_{H^1_{\alpha}}:= \int_0^1 u^2 \;\mathrm{d}x + \int_0^1 x^{\alpha} u_x^2 \;\mathrm{d}x \qquad \text{and} \qquad \|u\|^2_{H^2_{\alpha}}:= \|u\|^2_{H^1_{\alpha}} + \int_0^1 |(x^{\alpha} u_x)_x|^2 \;\mathrm{d}x.\]
Now, consider the operator
\begin{align*}
    \mathcal{A}u= (x^{\alpha} u_x)_x
\end{align*}
for $u\in D(\mathcal{A})$, where 
\begin{align*}
D(\mathcal{A})=\{  u\in H^2_\alpha(0,1) \, : \, \beta_0 u (0,t)- \beta_1 (x^{\alpha} u_x)(0,t)=0  \}.
\end{align*}
We recall the following  result that has been proven in \cite{NO'A}.
\begin{proposition}
     The operator $(\mathcal{A},D(\mathcal{A}))$ is self-adjoint, non-negative
with dense domain and generates a strongly continuous semigroup $(e^{t\mathcal{A}})_{t\geq 0}$ of contractions on $L^2(0, 1)$.
\end{proposition}
On the other hand, system mpulsive Cauchy problem \eqref{1.1}  can be written as the following impulsive Cauchy problem 
\begin{empheq}[left = \empheqlbrace]{alignat=2} \label{ACP}
\begin{aligned}
&\partial_t \psi(t) = \mathcal{A} \psi(t), && \quad \text{in } (0, T) \setminus \{\tau\}, \\
&\psi(\tau) = \psi(\tau^-) + 1_\omega h, && \\
&\psi(0) = \psi_0. &&
\end{aligned}
\end{empheq}
For $\psi_0\in L^2(0,1)$ , the system \eqref{ACP} has a unique mild solution given by
\begin{align*}
    \psi(t) = e^{t\mathcal{A}} \psi_0 + \mathbf{1}_{\{t \geq \tau\}}\, e^{(t - \tau)\mathcal{A}} 1_\omega(t) h, \quad t \in (0, T).
\end{align*}

\section{Logarithmic Convexity via the Carleman commutator approach}\label{section3}
In this section, we apply the Carleman commutator approach to prove the observation estimate at one point in time \eqref{ob}. To begin, we introduce the suitable weight function.
	
	\subsection{The weight function}
	Consider the smooth weight function $\phi$ given by
	\begin{equation}\label{phi}
		\phi(x,t)=s \,\dfrac{\varphi(x)}{\theta(t)},
	\end{equation}
	where \[\theta(t)=T-t+h \;\;\text{and}\;\; \varphi(x)=-\dfrac{1}{2(2-\alpha)^2} x^{2-\alpha} \;\;\text{for all} \;(x,t) \in (0,1)\times(0,T) ,\] 
    with $0<s<1$ and $h \in \left(0, 1 \right]$.

\subsection{Proof of the observation inequality}
\begin{proof}[Proof of Theorem \ref{thm1}]
The proof strategy for Lemma \ref{lem1}, inspired by the methodologies developed in \cite{BuffePhung2022, Phung2018}, unfolds as follows. Step 1 involves a change of variables and the introduction of symmetric and antisymmetric components. In Step 2, a frequency function adapted to our global framework is constructed. Assuming hypothesis \ref{hypo1} is satisfied, Steps 3 and 4 focus on computing the Carleman commutator and establishing essential estimates. Step 5 addresses the solution of a system of ordinary differential inequalities. Step 6 identifies the control region \( \omega \times T \), and Step 7 uses the specific form of the chosen weight function to derive the target estimate in equation \eqref{ob}.

\noindent \textbf{Step 1. Symmetric and antisymmetric parts.} Let us start by considering $u$ as the classical solution of the adjoint system \eqref{1.2}, and define the function
\[w= e^{\frac{\Phi}{2}} u, \qquad \text{for} \;s>0.\]
This yields the following system
  \begin{empheq}[left = \empheqlbrace]{alignat=2} \label{1.3}
		\begin{aligned}
			&\partial_t (w e^{-\frac{\Phi}{2}})(x,t) - (x^{\alpha} (w e^{-\frac{\Phi}{2}})_x)_x (x,t) = 0, && \quad\text { in } (0,1) \times(0, T) ,\\ 
            &\beta_0 (w e^{-\frac{\Phi}{2}}) (0,t) -\beta_1 (x^{\alpha} (w e^{-\frac{\Phi}{2}})_x)(0,t)=0, && \quad\text { in } (0, T) ,\\ 
            &w(1,t)=0, && \quad\text { in } (0, T) ,\\ 
			&w(x,0)= w_0(x),   && \quad \text{ on } (0,1).
		\end{aligned}
	\end{empheq}
It can be noted that the equation at the boundary $x=0$ takes the form
\begin{align}\label{Robine}
      \left(x^{\alpha} w_x\right)(0,t)=\dfrac{\beta_0}{\beta_1}w(0,t), \qquad \text{for all}\; t\in(0,T).
\end{align}
Now, let us introduce the operator \(L_1\) as follows
\begin{align*}
    L_1 w 
    := &\partial_t (w e^{-\frac{\Phi}{2}}) - (x^{\alpha} (w e^{-\frac{\Phi}{2}})_x)_x,
\end{align*}
such that
\[L_1 w=0.\]
Then
\begin{align*}
L_1 w=& w_t -\dfrac{1}{2} \Phi_t w-
\frac{1}{4}x^{\alpha}\Phi_{x}^2 w +  x^{\alpha} \Phi_{x} w_x + \frac{1}{2}(x^{\alpha}\Phi_{x})_{x} w - (x^{\alpha} w_{x})_{x}.
\end{align*}
Consider the operator $L$
\begin{equation*}
L w=\eta w  -  x^{\alpha} \Phi_{x} w_x - \frac{1}{2}(x^{\alpha}\Phi_{x})_{x} w+(x^{\alpha} w_{x})_{x} 
\end{equation*}
with $$\eta=\dfrac{1}{2} \Phi_t +\frac{1}{4}x^{\alpha}\Phi_{x}^2.$$
As result, one can obtain that
\[\partial_t w - L w=0.\]
Now, let us calculate the adjoint operator of \(L\). For any $ z\in H^1_{\alpha}(0,1)$, one has
\begin{align*}
    \left\langle L w, z\right\rangle
    =& \int_0^1 \eta w z \,\mathrm{d}x - \dfrac{1}{2} \int_0^1 (x^{\alpha} \Phi_x)_x w z \,\mathrm{d}x - \int_0^1 x^{\alpha} \Phi_x w_x z \,\mathrm{d}x + \int_0^1 (x^{\alpha} w_x)_x z \,\mathrm{d}x\\
    =&\left[x^{\alpha} w_x z - x^{\alpha} \Phi_x w z\right]_0^1 + \int_0^1\eta w z \,\mathrm{d}x+ \dfrac{1}{2} \int_0^1 (x^{\alpha} \Phi_x)_x w z \,\mathrm{d}x  + \int_0^1 x^{\alpha} \Phi_x z_x w \,\mathrm{d}x\\
    & -\int_0^1 x^{\alpha} w_x z_x \,\mathrm{d}x\\
    =&\left[x^{\alpha} w_x z - x^{\alpha} z_x w - x^{\alpha} \Phi_x w z\right]_0^1+\int_0^1 \eta w z \,\mathrm{d}x + \dfrac{1}{2} \int_0^1 (x^{\alpha} \Phi_x)_x w z \,\mathrm{d}x+ \int_0^1x^{\alpha} \Phi_x z_x w \,\mathrm{d}x\\
    & +\int_0^1 (x^{\alpha} z_x)_x w \,\mathrm{d}x.
\end{align*}
Due to the fact that $w(1,t)=z(1,t)=0$ for all $t \in (0,T)$, then
\[\left\lbrace x^{\alpha} w_x z - x^{\alpha} z_x w - x^{\alpha} \Phi_x w z \right\rbrace_{|x=1}=0.\]
Similarly, using \eqref{Robine} and $\beta_1 \neq 0$ which enables us to write the boundary terms at $x=0$ as follows
\begin{align*}
    \left\lbrace x^{\alpha} w_x z - x^{\alpha} z_x w - x^{\alpha} \Phi_x w z \right\rbrace_{|x=0}
    =&\left\lbrace x^{\alpha} w_x z - x^{\alpha} z_x w - \dfrac{s}{\theta} x^{\alpha} \varphi_x w z \right\rbrace_{|x=0}\\
    =&\left\lbrace x^{\alpha} w_x z - x^{\alpha} z_x w + \dfrac{s}{2(2-\alpha)\theta} x w z \right\rbrace_{|x=0}\\
    =&\left\lbrace \dfrac{\beta_0}{\beta_1 } w z - \dfrac{\beta_0}{\beta_1 } z w + \dfrac{s}{2(2-\alpha)\theta} x w z \right\rbrace_{|x=0}\\
    =&0.
\end{align*}
This implies that 
\[\left\langle L w, z\right\rangle=\left\langle  w, L^{*} z\right\rangle.\]
Thus, it follows that
\[L^{*}w = \eta w + \frac{1}{2} (x^{\alpha} \Phi_x)_x w + x^{\alpha} \Phi_x w_x +(x^{\alpha} w_x)_x\]
Based on this, we define the following operator
\[\mathcal{A}w:= \dfrac{L w  - L^{*}w}{2}= - x^{\alpha} \Phi_x w_x - \frac{1}{2} (x^{\alpha} \Phi_x)_x w\]
which is antisymmetric on $H^1_{\alpha}(0,1)$. Similarly, we define the operator 
\[\mathcal{S}w:= \dfrac{Lw  + L^{*}w}{2}= \eta w + (x^{\alpha} w_x)_w\]
that is symmetric on $H^1_{\alpha}(0,1)$. Thus,
\begin{equation}\label{eqqq1}
    \partial_t w = \mathcal{S}w + \mathcal{A}w.
\end{equation}

\noindent\textbf{Step 2. Energy estimates.}	Considering the antisymmetry of $\mathcal{A}$, one can obtain \[\left\langle \mathcal{A}w , w \right\rangle =0, \qquad \text{for all}\;\; w\in D(A).\]
    Moreover, by multiplying \eqref{eqqq1} by $w$ and integrating over $(0,1)$, one gets that for all $w\in D(A)$,
	\begin{equation}\label{eqqq3}
		\dfrac{1}{2} \partial_t \parallel w \parallel^2 + \left\langle -\mathcal{S}w , w \right\rangle =0.
	\end{equation}
	We define the frequency function 
	\begin{equation*}
		\mathcal{N}(t)=\dfrac{\left\langle - \mathcal{S}w ,w \right\rangle}{\parallel w \parallel^2}.
	\end{equation*}
	Thus, \eqref{eqqq3} becomes
	\begin{equation}\label{eqqq4}
		\dfrac{1}{2} \partial_t \parallel w \parallel^2 + \mathcal{N}(t) \parallel w \parallel^2 = 0.
	\end{equation}
	Now, we continue by estimating the derivative of the frequency function $\mathcal{N}$ as follows
    \begin{align}\label{eqqq5}
        \partial_t \mathcal{N}(t) \parallel w \parallel^2 
        \leq &\left\langle -\left(\mathcal{S}'  + [\mathcal{S} , \mathcal{A} ] \right) w , w \right\rangle\nonumber\\
        &+\left[x^{\alpha} (\mathcal{A} w)_x w- x^{\alpha} w_x (\mathcal{A} w)-  x^{\alpha} \Phi_x (\mathcal{S} w) w\right]_0^1.
    \end{align}
    To derive this estimate, we begin by evaluating \(\partial_t \left\langle \mathcal{S}w , w \right\rangle\). Next, by integrating by parts, we have
    \begin{align}\label{eqc}
        \left\langle \mathcal{S}w , w \right\rangle
        = &\int_0^1  \left( (x^{\alpha} w_x)_x + \eta w \right) w \,\mathrm{d}x\nonumber\\
        =&\left[ x^{\alpha} w_x w\right]_0^1-\int_0^1 x^{\alpha} w_x^2\,\mathrm{d}x + \int_0^1 \eta w^2 \,\mathrm{d}x\nonumber\\
        =&-\{ x^{\alpha} w_x w\}_{|x=0}-\int_0^1 x^{\alpha} w_x^2\,\mathrm{d}x + \int_0^1 \eta w^2 \,\mathrm{d}x\nonumber\\
        =&-\{ \dfrac{\beta_0}{\beta_1} w^2\}_{|x=0}-\int_0^1 x^{\alpha} w_x^2\,\mathrm{d}x + \int_0^1 \eta w^2 \,\mathrm{d}x.
    \end{align}
    Hence, we infer that
    \begin{align*}
        &\partial_t \left\langle \mathcal{S}w , w \right\rangle\\
        =&-2 \dfrac{\beta_0}{\beta_1} \{ \partial_t w w\}_{|x=0} -2\int_0^1 x^{\alpha} \partial_t w_x w_x \,\mathrm{d}x + \int_0^1 \partial_t \eta w^2 \,\mathrm{d}x + 2 \int_0^1 \eta \partial_t w w \,\mathrm{d}x\\
        =&-2 \dfrac{\beta_0}{\beta_1} \{ \partial_t w w\}_{|x=0} -2 \left[ x^{\alpha} w_x \partial_t w\right]_0^1 + 2\int_0^1 (x^{\alpha} w_x)_x \partial_t w  \,\mathrm{d}x+2 \int_0^1 \eta \partial_t w w \,\mathrm{d}x + \int_0^1 \partial_t \eta w^2 \,\mathrm{d}x \\
        =&-2 \dfrac{\beta_0}{\beta_1} \{ \partial_t w w\}_{|x=0} -2 \left\lbrace x^{\alpha} w_x \partial_t w\right\rbrace_{|x=1} + 2 \dfrac{\beta_0}{\beta_1} \left\lbrace \partial_t w w \right\rbrace_{|x=0} + 2\int_0^1 (x^{\alpha} w_x)_x \partial_t w  \,\mathrm{d}x+2 \int_0^1 \eta \partial_t w w \,\mathrm{d}x\\
        &+ \int_0^1 \partial_t \eta w^2 \,\mathrm{d}x \\
        =&2\left\langle \mathcal{S} w , \partial_t w\right\rangle + \int_0^1 \partial_t \eta w^2 \,\mathrm{d}x.
    \end{align*}
    From the definition \(\mathcal{S}' w := \partial_t \eta w\) and the previous equation, we obtain
    \begin{equation*}
        \partial_t \left\langle \mathcal{S}w , w \right\rangle= 2\left\langle \mathcal{S}w , \partial_t w \right\rangle + \left\langle \mathcal{S}' w , w \right\rangle.
    \end{equation*}
    By combining this result with \eqref{eqqq1} and applying the Cauchy-Schwarz inequality, we proceed to
	\begin{align}\label{eqqq9}
		&\partial_t \mathcal{N}(t) \parallel w \parallel^4\nonumber\\
		=& \partial_t \left\langle - \mathcal{S}w , w \right\rangle \parallel w \parallel^2 + 2 \left\langle \mathcal{S}w , w \right\rangle \left\langle \partial_t w , w \right\rangle  \nonumber\\
        =& \left( \left\langle -\mathcal{S}' w, w \right\rangle - 2 \left\langle \mathcal{S}w , \partial_t w \right\rangle \right) \parallel w \parallel^2 + 2
		\left\langle \mathcal{S}w , w \right\rangle \left\langle \partial_t w , w \right\rangle \nonumber\\
        = &\left( \left\langle -\mathcal{S}' w, w \right\rangle - 2 \left\langle \mathcal{S}w , \mathcal{A}w \right\rangle \right) \parallel w \parallel^2 + 2\left( \left\lvert\left\langle \mathcal{S}w , w \right\rangle \right\rvert^2 -\| \mathcal{S}w \|^2 \parallel w \parallel^2\right)\nonumber\\
        \leq &\left( \left\langle -\mathcal{S}' w, w \right\rangle - 2 \left\langle \mathcal{S}w , \mathcal{A}w \right\rangle \right) \parallel w \parallel^2.
    \end{align}
	Next, we assert the following equality
	\begin{align}\label{starequality1}
		2 \left\langle \mathcal{S} w, \mathcal{A }w \right\rangle 
        =& \left\langle [\mathcal{S} , \mathcal{A} ]w , w\right\rangle +\left[ x^{\alpha} w_x (\mathcal{A} w)- x^{\alpha} (\mathcal{A} w)_x w + x^{\alpha} \Phi_x (\mathcal{S} w) w\right]_0^1,
	\end{align}
    where the commutator is defined as 
	\[\left[\mathcal{S} , \mathcal{A} \right]w:= \mathcal{S A}w - \mathcal{A S}w.\]
	Indeed, using the definition of \(\mathcal{S}w\) and \(\mathcal{A}w\) one has by integrating by parts that
	\begin{align}\label{eqqq6}
		&\left\langle \mathcal{S A} w ,w \right\rangle\nonumber\\
		=&\int_0^1 \left(\eta \mathcal{A} w + (x^{\alpha} (\mathcal{A} w)_x)_x\right) w \,\mathrm{d}x\nonumber\\
		=&\left[ x^{\alpha} (\mathcal{A} w)_x w \right]_0^1+\int_0^1 \eta (\mathcal{A} w) w \,\mathrm{d}x - \int_0^1 x^{\alpha} (\mathcal{A} w)_x w_x \,\mathrm{d}x \nonumber\\
		=&\left[ x^{\alpha} (\mathcal{A} w)_x w - x^{\alpha} w_x (\mathcal{A} w) \right]_0^1+\int_0^1 \eta w (\mathcal{A} w) \,\mathrm{d}x + \int_0^1 (x^{\alpha} w_x)_x (\mathcal{A} w) \,\mathrm{d}x \nonumber\\
        =&\left[ x^{\alpha} (\mathcal{A} w)_x w - x^{\alpha} w_x (\mathcal{A} w) \right]_0^1+ \left\langle \mathcal{S} W ,\mathcal{A}W \right\rangle.
	\end{align}
	On the other hand, one has
	\begin{align}\label{eqqq7}
		&\left\langle \mathcal{A S}w, w\right\rangle\nonumber\\
		=&\int_0^1 \left((-x^{\alpha} \Phi_x (\mathcal{S} w)_x -\dfrac{1}{2} (x^{\alpha} \Phi_x)_x (\mathcal{S} w) \right) w\,\mathrm{d}x \nonumber\\
        =&\left[ -x^{\alpha} \Phi_x (\mathcal{S} w) w\right]_0^1+\int_0^1 x^{\alpha} \Phi_x w_x (\mathcal{S} w) \,\mathrm{d}x + \dfrac{1}{2} \int_0^1 (x^{\alpha} \Phi_x)_x (\mathcal{S} w) w\,\mathrm{d}x\nonumber\\
        =&\left[ -x^{\alpha} \Phi_x (\mathcal{S} w) w\right]_0^1-\left\langle \mathcal{S}w, \mathcal{A} w \right\rangle .
	\end{align}
	By combining \eqref{eqqq6} and \eqref{eqqq7}, the result is established.\\
 
    Finally, using \eqref{eqqq9} and \eqref{starequality1}, we can conclude the desired inequality \eqref{eqqq5}.\\

	\noindent\textbf{Step 3. Calculating the Carleman commutator.} Next, we move on to compute the Carleman commutator $\left[\mathcal{S} , \mathcal{A} \right]W$, beginning with
	\begin{align}\label{AS}
		\mathcal{A S}w
        =&-x^{\alpha} \Phi_x \left(\mathcal{S} w\right)_x - \dfrac{1}{2} (x^{\alpha} \Phi_x)_x (\mathcal{S} w)\nonumber\\
        =&-x^{\alpha} \Phi_x \left(\eta w + (x^{\alpha} w_x)_x\right)_x - \dfrac{1}{2} (x^{\alpha} \Phi_x)_x (x^{\alpha} w_x)_x-\dfrac{1}{2} \eta (x^{\alpha} \Phi_x)_x w \nonumber\\
        =&-x^{\alpha} \Phi_x (x^{\alpha} w_x)_{xx} - \dfrac{1}{2} (x^{\alpha} \Phi_x)_x (x^{\alpha} w_x)_x-\eta x^{\alpha} \Phi_x w_x - \dfrac{1}{2} (x^{\alpha} \Phi_x)_x \eta w - x^{\alpha} \Phi_x \eta_x w.
	\end{align}
	In a similar manner, one has
	\begin{align}\label{SA}
		\mathcal{S A}w
		=&\eta\left(\mathcal{A} w\right) + (x^{\alpha} (\mathcal{A} w)_x)_x\nonumber\\
        =&-\left( x^{\alpha} \left( x^{\alpha} \Phi_x w_x + \dfrac{1}{2} (x^{\alpha} \Phi_x)_x w \right)_x\right)_x - \eta x^{\alpha} \Phi_x w_x -\dfrac{1}{2} \eta (x^{\alpha} \Phi_x)_x w.
	\end{align}
	Combining this with the expression for \eqref{AS} in \eqref{AS}, we derive
	\begin{align}\label{eqqq10}
		[\mathcal{S} , \mathcal{A}]w
		=&-\left( x^{\alpha} \left( x^{\alpha} \Phi_x w_x + \dfrac{1}{2} (x^{\alpha} \Phi_x)_x w \right)_x\right)_x +x^{\alpha} \Phi_x (x^{\alpha} w_x)_{xx} + \dfrac{1}{2} (x^{\alpha} \Phi_x)_x (x^{\alpha} w_x)_x + x^{\alpha} \Phi_x \eta_x w. 
	\end{align}
	From the provided definition of $\mathcal{S}' W$, one has
    \begin{align}
        \mathcal{S}' w
        =&\eta_t w.
    \end{align}
    Combining this expression with the identity \eqref{eqqq10} for the commutator $[\mathcal{S} , \mathcal{A}]$, we arrive at the following formulation
	\begin{align}\label{eq16}
		&-(S'+[\mathcal{S} , \mathcal{A}])w\nonumber\\
		=&\left( x^{\alpha} \left( x^{\alpha} \Phi_x w_x + \dfrac{1}{2} (x^{\alpha} \Phi_x)_x w \right)_x\right)_x -x^{\alpha} \Phi_x (x^{\alpha} w_x)_{xx} - \dfrac{1}{2} (x^{\alpha} \Phi_x)_x (x^{\alpha} w_x)_x - (x^{\alpha} \Phi_x \eta_x + \partial_t \eta) w. 
	\end{align}
	Next, we concentrate on the boundary terms. In fact, it holds that
	\begin{align}\label{eq17}
		&\left[x^{\alpha} (\mathcal{A} w)_x w- x^{\alpha} w_x (\mathcal{A} w)-  x^{\alpha} \Phi_x (\mathcal{S} w) w\right]_0^1\nonumber\\
        =&\left\lbrace x^{\alpha} (\mathcal{A} w)_x w- x^{\alpha} w_x (\mathcal{A} w)-  x^{\alpha} \Phi_x (\mathcal{S} w) w \right\rbrace_{|x=1}- \left\lbrace x^{\alpha} (\mathcal{A} w)_x w- x^{\alpha} w_x (\mathcal{A} w)-  x^{\alpha} \Phi_x (\mathcal{S} w) w \right\rbrace_{|x=0}\nonumber\\
        =&-\left\lbrace x^{\alpha} \left(x^{\alpha} \Phi_x w_x + \dfrac{1}{2}(x^{\alpha} \Phi_x)_x w\right)_x w- x^{\alpha} w_x \left(x^{\alpha} \Phi_x w_x +\dfrac{1}{2}(x^{\alpha} \Phi_x)_x w\right)\right.\nonumber\\
        &\left.+  x^{\alpha} \Phi_x \left((x^{\alpha} w_x)_x + \eta w\right) w \right\rbrace_{|x=1} +\left\lbrace x^{\alpha} \left(x^{\alpha} \Phi_x w_x + \dfrac{1}{2}(x^{\alpha} \Phi_x)_x w\right)_x w\right.\nonumber\\
        &\left.- x^{\alpha} w_x \left(x^{\alpha} \Phi_x w_x +\dfrac{1}{2}(x^{\alpha} \Phi_x)_x w\right)+  x^{\alpha} \Phi_x \left((x^{\alpha} w_x)_x + \eta w\right) w \right\rbrace_{|x=0}.
	\end{align}
	The two equalities \eqref{eq16} and \eqref{eq17} imply that
	\begin{align}\label{eq18}
		&\left\langle -(S'+[\mathcal{S} , \mathcal{A}])w , w \right\rangle+\left[x^{\alpha} (\mathcal{A} w)_x w- x^{\alpha} w_x (\mathcal{A} w)-  x^{\alpha} \Phi_x (\mathcal{S} w) w\right]_0^1\nonumber\\
		=&-\left\lbrace x^{\alpha} \left(x^{\alpha} \Phi_x w_x + \dfrac{1}{2}(x^{\alpha} \Phi_x)_x w\right)_x w- x^{\alpha} w_x \left(x^{\alpha} \Phi_x w_x +\dfrac{1}{2}(x^{\alpha} \Phi_x)_x w\right)\right.\nonumber\\
        &\left.+  x^{\alpha} \Phi_x \left((x^{\alpha} w_x)_x + \eta w\right) w \right\rbrace_{|x=1} +\left\lbrace x^{\alpha} \left(x^{\alpha} \Phi_x w_x + \dfrac{1}{2}(x^{\alpha} \Phi_x)_x w\right)_x w\right.\nonumber\\
        &\left.- x^{\alpha} w_x \left(x^{\alpha} \Phi_x w_x +\dfrac{1}{2}(x^{\alpha} \Phi_x)_x w\right)+  x^{\alpha} \Phi_x \left((x^{\alpha} w_x)_x + \eta w\right) w \right\rbrace_{|x=0}\nonumber\\
		&+ \int_0^1 \left( x^{\alpha} \left( x^{\alpha} \Phi_x w_x + \dfrac{1}{2} (x^{\alpha} \Phi_x)_x w \right)_x\right)_x w \,\mathrm{d}x - \int_0^1  x^{\alpha} \Phi_x (x^{\alpha} w_x)_{xx} w \,\mathrm{d}x\nonumber\\
        &- \dfrac{1}{2} \int_0^1  (x^{\alpha} \Phi_x)_x (x^{\alpha} w_x)_x w \,\mathrm{d}x - \int_0^1  \left(x^{\alpha} \Phi_x \eta_x + \eta_t\right) w^2 \,\mathrm{d}x.
	\end{align}
	To enhance the clarity of the previous identity, we begin by performing integration by parts. First, we observe that
	\begin{align}\label{ip1}
		&\int_0^1 \left( x^{\alpha} \left( x^{\alpha} \Phi_x w_x + \dfrac{1}{2} (x^{\alpha} \Phi_x)_x w \right)_x\right)_x w \,\mathrm{d}x\nonumber\\
		=&\left[ \,x^{\alpha} \left( x^{\alpha} \Phi_x w_x + \dfrac{1}{2} (x^{\alpha} \Phi_x)_x w \right)_x w\, \right]_0^1 - \int_0^1 x^{\alpha} w_x \left( x^{\alpha} \Phi_x w_x + \dfrac{1}{2} (x^{\alpha} \Phi_x)_x w \right)_x  \,\mathrm{d}x\nonumber\\
        =&\left\lbrace x^{\alpha}  \left( x^{\alpha} \Phi_x w_x + \dfrac{1}{2} (x^{\alpha} \Phi_x) w \right)_x w\, \right\rbrace_{x=1} - \left\lbrace x^{\alpha}  \left( x^{\alpha} \Phi_x w_x + \dfrac{1}{2} (x^{\alpha} \Phi_x) w \right)_x w\, \right\rbrace_{x=0} \nonumber\\
        &- \int_0^1 x^{\alpha} \phi_x (x^{\alpha} w_x)_x w_x \,\mathrm{d}x- \dfrac{1}{2} \int_0^1 (x^{\alpha} \phi_x)_{xx} x^{\alpha} w_x w \,\mathrm{d}x - \int_0^1 \left( \,x^{\alpha} \phi_{xx} + \dfrac{1}{2} (x^{\alpha} \phi_x)_x \,\right) x^{\alpha} w_x^2 \,\mathrm{d}x.
    \end{align}
	Secondly, one has
	\begin{align}\label{ip2}
		&\int_0^1  x^{\alpha} \Phi_x (x^{\alpha} w_x)_{xx} w \,\mathrm{d}x\nonumber\\
		=& \left[ \,x^{\alpha} \Phi_x (x^{\alpha} w_x)_x w \right]_0^1 - \int_0^1 (x^{\alpha} \Phi_x)_x (x^{\alpha} w_x)_x w \,\mathrm{d}x - \int_0^1 x^{\alpha} \Phi_x (x^{\alpha} w_x)_x w_x \,\mathrm{d}x\nonumber\\
        =&\left\lbrace x^{\alpha} \Phi_x (x^{\alpha} w_x)_x w \right\rbrace_{|x=1} - \left\lbrace x^{\alpha} \Phi_x (x^{\alpha} w_x)_x w \right\rbrace_{|x=0}  - \int_0^1 (x^{\alpha} \Phi_x)_x (x^{\alpha} w_x)_x w \,\mathrm{d}x\nonumber\\
        &- \int_0^1 x^{\alpha} \Phi_x (x^{\alpha} w_x)_x w_x \,\mathrm{d}x.
	\end{align}
	Using \eqref{eq18}, \eqref{ip1}, and \eqref{ip2}, along with the boundary condition $w(1,t)=0$ for any $t\in (0,T)$, we reach that
	\begin{align}\label{eq19}
		&\left\langle -(S'+[\mathcal{S} , \mathcal{A}])w , w \right\rangle+\left[x^{\alpha} (\mathcal{A} w)_x w- x^{\alpha} w_x (\mathcal{A} w)-  x^{\alpha} \Phi_x (\mathcal{S} w) w\right]_0^1\nonumber\\
		=&\left\lbrace -x^{\alpha} \Phi_x x^{\alpha} w_x^2 - \dfrac{1}{2} (x^{\alpha} \Phi_x)_x x^{\alpha} w_x w  + \eta x^{\alpha} \Phi_x w^2 \right\rbrace_{|x=0}+\left\lbrace x^{\alpha} \Phi_x x^{\alpha} w_x^2 \right\rbrace_{|x=1}\nonumber\\
		&+\dfrac{1}{2} \int_0^1 ( x^{\alpha} \Phi_x)_x ( x^{\alpha} w_x)_x w \,\mathrm{d}x - \dfrac{1}{2} \int_0^1  (x^{\alpha} \Phi_x)_{xx} x^{\alpha} w_x w \,\mathrm{d}x\nonumber\\
        &- \int_0^1  \left(x^{\alpha} \Phi_{xx} + \dfrac{1}{2} (x^{\alpha} \Phi_x)_x \right) x^{\alpha} w_x^2 \,\mathrm{d}x - \int_0^1  \left(x^{\alpha} \Phi_x \eta_x + \eta_t\right) w^2 \,\mathrm{d}x.
	\end{align}
    Thirdly, through the application of integrations by parts one can derive that
	\begin{align*}
		&\dfrac{1}{2} \int_0^1 ( x^{\alpha} \Phi_x)_x ( x^{\alpha} w_x)_x w \,\mathrm{d}x \nonumber\\
		=& \left[ \,\dfrac{1}{2}( x^{\alpha} \Phi_x)_x x^{\alpha} w_x w\, \right]_0^1 - \dfrac{1}{2} \int_0^1 ( x^{\alpha} \Phi_x)_{xx} x^{\alpha} w_x w \,\mathrm{d}x - \dfrac{1}{2} \int_0^1( x^{\alpha} \Phi_x)_x x^{\alpha} w_x^2 \,\mathrm{d}x.
    \end{align*}
     Inserting the above equality into \eqref{eq19} gives
	\begin{align*}
		&\left\langle -(S'+[\mathcal{S} , \mathcal{A}])w , w \right\rangle_{\mathbb{L}^2} +\left[x^{\alpha} (\mathcal{A} w)_x w- x^{\alpha} w_x (\mathcal{A} w)-  x^{\alpha} \Phi_x (\mathcal{S} w) w\right]_0^1\nonumber\\
		=&\left\lbrace -x^{\alpha} \Phi_x x^{\alpha} w_x^2 -  (x^{\alpha} \Phi_x)_x x^{\alpha} w_x w  + \eta x^{\alpha} \Phi_x w^2 + ( x^{\alpha} \Phi_x)_{xx} x^{\alpha} w^2\right\rbrace_{|x=0}\nonumber\\
        &+\left\lbrace x^{\alpha} \Phi_x x^{\alpha} w_x^2 \right\rbrace_{|x=1}+ \int_0^1  \left(x^{\alpha} (x^{\alpha} \Phi_x)_{xx}\right)_x  w^2 \,\mathrm{d}x - \int_0^1  \left(x^{\alpha} \Phi_{xx} + (x^{\alpha} \Phi_x)_x \right) x^{\alpha} w_x^2 \,\mathrm{d}x\nonumber\\
		&- \int_0^1  \left(x^{\alpha} \Phi_x \eta_x + \eta_t\right) w^2 \,\mathrm{d}x.
	\end{align*}
    On the other hand, the definition of $\eta$ enables us to write
    \begin{align*}
        \partial_t \eta + x^{\alpha} \phi_x \eta_x
        = &\dfrac{1}{2} \left( \partial_{tt} \phi + x^{\alpha} \phi_x \partial_t \phi_x \right) + \dfrac{1}{2} x^{\alpha} \phi_x \left( \partial_t \phi_x + \dfrac{1}{2} (x^{\alpha} \phi_x)_x \phi_x + \dfrac{1}{2} x^{\alpha} \phi_{xx} \phi_x \right)\\
        = &\dfrac{1}{2} \left( \partial_{tt} \phi + 2 x^{\alpha} \phi_x \partial_t \phi_x + \dfrac{1}{2} x^{\alpha} \phi_x^2 \left( (x^{\alpha} \phi_x)_x + x^{\alpha} \phi_{xx} \right) \right).
    \end{align*}
	Finally, it can be inferred from the definition of the weight function $\phi$ that
    \begin{align}\label{eqqq11}
		&\left\langle -(S'+[\mathcal{S} , \mathcal{A}])w , w \right\rangle +\left[x^{\alpha} (\mathcal{A} w)_x w- x^{\alpha} w_x (\mathcal{A} w)-  x^{\alpha} \Phi_x (\mathcal{S} w) w\right]_0^1\nonumber\\
		=&\dfrac{s}{\theta}\left\lbrace -x^{\alpha} \varphi_x x^{\alpha} w_x^2 -  (x^{\alpha} \varphi_x)_x x^{\alpha} w_x w   + ( x^{\alpha} \varphi_x)_{xx} x^{\alpha} w^2\right\rbrace_{|x=0}+\dfrac{s^2}{2\theta^3}\left\lbrace \left(\varphi  + \dfrac{s}{2} x^{\alpha} \varphi_x^2 \right)x^{\alpha} \varphi_x w^2 \right\rbrace_{|x=0}\nonumber\\
        &+\dfrac{s}{\theta}\left\lbrace x^{\alpha} \varphi_x x^{\alpha} w_x^2 \right\rbrace_{|x=1} + \dfrac{s}{2\theta} \int_0^1  \left(x^{\alpha}(x^{\alpha} \varphi_x)_{xx}\right)_x  w^2 \,\mathrm{d}x - \dfrac{s}{\theta}\int_0^1  \left(x^{\alpha} \varphi_{xx} + (x^{\alpha} \varphi_x)_x \right) x^{\alpha} w_x^2 \,\mathrm{d}x\nonumber\\
        &- \dfrac{s}{2\theta^3} \int_0^1  \left(2 \varphi + 2 s x^{\alpha} \varphi_x^2 + \dfrac{s^2}{2} x^{\alpha} \varphi_x^2 \left( (x^{\alpha} \varphi_x)_x + x^{\alpha} \varphi_{xx}\right)\right) w^2 \,\mathrm{d}x.
	\end{align}

\noindent\textbf{Step 4. Carleman commutator estimates.} In this step, we assert that for any $h\in (0, 1]$ along with the given choice $0<s<1$, the following inequality holds
	\begin{align}\label{eq22}
		&\left\langle -(S'+[\mathcal{S} , \mathcal{A}])w , w \right\rangle +\left[x^{\alpha} (\mathcal{A} w)_x w- x^{\alpha} w_x (\mathcal{A} w)-  x^{\alpha} \Phi_x (\mathcal{S} w) w\right]_0^1 \nonumber\\ 
        &\leq \dfrac{1+C_0}{\theta} \left\langle -\mathcal{S}w , w \right\rangle.
	\end{align}
Here, the constant $C_0(s)\in (0, 1)$ is independent of $h\in (0, 1]$. To validate this claim, using the definition of the function $\varphi$, we can easily verify the following results
	\begin{enumerate}
		\item $(x^{\alpha} \varphi_x)_{xx}=0$,
		\item $x^{\alpha} \varphi_x \left(\dfrac{s^2}{2} x^{\alpha} \varphi_x^2  +  \varphi\right)=\dfrac{(4-s)}{16(2-\alpha)^3} x^{3-\alpha}$,
        \item $(x^{\alpha} \varphi_x)_x + x^{\alpha} \varphi_{xx} = -\dfrac{1}{2}$,
		\item $ 2\varphi + 2s \,x^{\alpha} \varphi_x^2 + \dfrac{s^2}{4} x^{\alpha} \varphi_x^2 \left( (x^{\alpha} \varphi_x )_x  + x^{\alpha}\varphi_{xx}\right) =- \dfrac{(4-s)^2}{16(2-\alpha)^2} \,x^{2-\alpha}$.
	\end{enumerate}
	Consequently, by taking the above quantities into account, \eqref{eqqq11} becomes
	\begin{align}\label{ing1}
		&\left\langle -(S'+[\mathcal{S} , \mathcal{A}])w , w \right\rangle +\left[x^{\alpha} (\mathcal{A} w)_x w- x^{\alpha} w_x (\mathcal{A} w)-  x^{\alpha} \Phi_x (\mathcal{S} w) w\right]_0^1 \nonumber\\
		=&\dfrac{s}{2(2-\alpha)\theta}\left\lbrace x^{1+\alpha} w_x^2 + x^{\alpha} w_x w \right\rbrace_{|x=0} +\dfrac{s^2}{2 \theta^3}\left\lbrace  \dfrac{(4-s)}{16(2-\alpha)^3} x^{3-\alpha} w^2 \right\rbrace_{|x=0} \nonumber\\
        & -\dfrac{s}{2(2-\alpha)\theta}\left\lbrace x^{1+\alpha} w_x^2  \right\rbrace_{|x=1} 
		+ \dfrac{s}{2\theta}\int_0^1  x^{\alpha} w_x^2 \,\mathrm{d}x + \dfrac{s(4-s)^2}{32(2-\alpha)^2\theta^3} \int_0^1  x^{2-\alpha} w^2 \,\mathrm{d}x.
	\end{align}
	Now, let us assess the contribution of the boundary terms. Indeed, due to \(w \in D(A)\) one can observe that \eqref{Robine} holds true. Thus, one finds
	\begin{align}\label{eq25}
		\left\lbrace x^{1+\alpha} w_x^2 + x^{\alpha} w_x w \right\rbrace_{|x=0}
		=&\left\lbrace x^{1-\alpha} \left(\dfrac{\beta_0}{\beta_1}  w\right)^2 + \dfrac{\beta_0}{\beta_1}  w^2  \right\rbrace_{|x=0} .
	\end{align}
    As shown in \cite{Alabau-Boussouira2006}, we have the following limits hold
    \[\lim_{x \to 0} x^{3-\alpha}  w^2(x)=0 \qquad \text{and} \qquad \lim_{x \to 0} x^{1-\alpha} w^2(x)=0.\]
    As a result,
    \[\dfrac{s}{2(2-\alpha)\theta}\left\lbrace x^{1+\alpha} w_x^2 + x^{\alpha} w_x w \right\rbrace_{|x=0} +\dfrac{s^2}{2 \theta^3}\left\lbrace  \dfrac{(4-s)}{16(2-\alpha)^3} x^{3-\alpha} w^2 \right\rbrace_{|x=0}=\dfrac{\beta_0}{\beta_1}\dfrac{s}{2(2-\alpha)\theta} \{w^2\}_{|x=0}.\]
	Using the fact that $w(1,\cdot)=0$, one can observe that all terms vanish except for $\{- x^{1+\alpha} w_x^2\}$ at $x=1$, which is defined and negative. Consequently, \eqref{ing1} becomes
	\begin{equation}\label{eq28}
		\left\langle -(S'+[\mathcal{S} , \mathcal{A}])w , w \right\rangle +\left[x^{\alpha} (\mathcal{A} w)_x w- x^{\alpha} w_x (\mathcal{A} w)-  x^{\alpha} \Phi_x (\mathcal{S} w) w\right]_0^1 \leq \mathcal{Q}w,
	\end{equation}
	where
    \[\mathcal{Q}w :=\dfrac{\beta_0}{\beta_1}\dfrac{s}{2(2-\alpha)\theta} \{w^2\}_{|x=0}+ \dfrac{s}{2\theta}\int_0^1  x^{\alpha} w_x^2 \,\mathrm{d}x + \dfrac{s(4-s)^2}{32(2-\alpha)^2\theta^3} \int_0^1  x^{2-\alpha} w^2 \,\mathrm{d}x\]
    On the other hand, by considering \eqref{eqc} and the definition of $\eta$, one can see that
	\begin{align*}
		\left\langle \mathcal{S}w , w \right\rangle
		&=-\{ \dfrac{\beta_0}{\beta_1} w^2\}_{|x=0}-\int_0^1 x^{\alpha} w_x^2\,\mathrm{d}x + \int_0^1 \eta w^2 \,\mathrm{d}x \,\mathrm{d}x\nonumber\\
        &=-\{ \dfrac{\beta_0}{\beta_1} w^2\}_{|x=0}-\int_0^1 x^{\alpha} w_x^2\,\mathrm{d}x + \dfrac{s}{2 \theta^2}\int_0^1 \left( \varphi + \dfrac{s}{2} x^{\alpha} \varphi_x^2 \right) w^2 \,\mathrm{d}x \,\mathrm{d}x\nonumber\\
		&= -\{ \dfrac{\beta_0}{\beta_1} w^2\}_{|x=0}-\int_0^1 x^{\alpha} w_x^2\,\mathrm{d}x - \dfrac{s(4-s)}{16(2-\alpha)^2\theta^2} \int_0^1 x^{2-\alpha} w^2 \,\mathrm{d}x,
	\end{align*}
    which implies that
	\begin{equation}\label{ineg2}
		\dfrac{1}{\theta} \left\langle -\mathcal{S}w , w \right\rangle
		=\dfrac{1}{\theta}\left\lbrace \dfrac{\beta_0}{\beta_1} w^2\right\rbrace_{|x=0}+\dfrac{1}{\theta}\int_0^1 x^{\alpha} w_x^2\,\mathrm{d}x + \dfrac{s(4-s)}{16(2-\alpha)^2\theta^3} \int_0^1 x^{2-\alpha} w^2 \,\mathrm{d}x.
	\end{equation}
	Next, by considering the selected choice of $s$, it follows that we can write
	\begin{align*}
		\mathcal{Q}w
        =&\dfrac{1}{\theta}\left(\dfrac{\beta_0}{\beta_1}\dfrac{s}{2(2-\alpha)} \{w^2\}_{|x=0}+ \dfrac{s}{2}\int_0^1  x^{\alpha} w_x^2 \,\mathrm{d}x + \dfrac{s(4-s)^2}{32(2-\alpha)^2\theta^2} \int_0^1  x^{2-\alpha} w^2 \,\mathrm{d}x\right)\\
         =&\dfrac{4-s}{2\theta}\left(\dfrac{\beta_0}{\beta_1}\dfrac{s}{(2-\alpha)(4-s)} \{w^2\}_{|x=0}+ \dfrac{s}{4-s}\int_0^1  x^{\alpha} w_x^2 \,\mathrm{d}x + \dfrac{s(4-s)}{16(2-\alpha)^2\theta^2} \int_0^1  x^{2-\alpha} w^2 \,\mathrm{d}x\right)\\
         \leq&\dfrac{4-s}{2\theta}\left(\dfrac{\beta_0}{\beta_1} \{w^2\}_{|x=0}+ \int_0^1  x^{\alpha} w_x^2 \,\mathrm{d}x + \dfrac{s(4-s)}{16(2-\alpha)^2\theta^2} \int_0^1  x^{2-\alpha} w^2 \,\mathrm{d}x\right)\\
	\end{align*}
	Therefore, from \eqref{ineg2} and the previous estimate, it follows that
	\begin{equation}\label{eq35}
		\mathcal{Q}w \leq \dfrac{1+C_0}{\theta} \left\langle -\mathcal{S}w , w \right\rangle,
	\end{equation}
	with $C_0=\left(1-\dfrac{s}{2}\right)\in(0,1)$. When this is paired with \eqref{eq28}, it leads to the desired inequality \eqref{eq22}.

    \noindent\textbf{Step 5. Solving a system of ordinary differential inequalities.} The following differential inequalities hold
	\begin{empheq}[left = \empheqlbrace]{alignat=2} \label{1.4}
		\begin{aligned}
			& \dfrac{1}{2} \partial_t \parallel w(\cdot,t) \parallel^2 + \mathcal{N}(t) \parallel w(\cdot,t) \parallel^2  =0,\\
			& \partial_t \mathcal{N}(t) \leq \dfrac{1+C_0}{\theta} \mathcal{N}(t) .
		\end{aligned}
	\end{empheq}
	By using \cite[Proposition 3]{BuffePhung2022}, we solve the system of differential inequalities above and find that for any $0< t_1 < t_2 < t_3 \leq T$, one has
	\begin{equation*}
		\parallel w(\cdot,t_2) \parallel^{1+M} \leq \parallel w(\cdot,t_1) \parallel^M \parallel w(\cdot,t_3) \parallel,
	\end{equation*}
	where
	\begin{equation*}
		M=\dfrac{\displaystyle{\int_{t_2}^{t_3}} \dfrac{1}{(T-t+h)^{1+C_0}} \,dt}{\displaystyle{\int_{t_1}^{t_2}} \dfrac{1}{(T-t+h)^{1+C_0}} \,dt}.
	\end{equation*}
	In other words, we find that
	\begin{equation}\label{eqq5}
		\parallel u(\cdot,t_2) \,e^{\frac{\phi}{2}(\cdot,t_2)} \parallel^{1+M} \leq \parallel u(\cdot,t_1) \,e^{\frac{\phi}{2}(\cdot,t_1)} \parallel^M \parallel u(\cdot,t_3) \,e^{\frac{\phi}{2}(\cdot,t_3)} \parallel.
	\end{equation}

    \noindent\textbf{Step 6. Making appear the control domain $\omega\times\{T\}$.} In this step, we eliminate the weight function $\phi$ from the integrals in the inequality \eqref{eqq5} and introduce the local term. To proceed, we first establish the following lemma, which will prove useful.
	\begin{lemma}\label{lem3}
		Assume that Hypothesis \ref{hypo1} holds true. For any solution $u$ of the system \eqref{1.2}, the function $t \mapsto  \parallel u(\cdot,t) \parallel$ is decreasing on $(0,T)$.
	\end{lemma}
	
	\begin{proof}
		Through integration by parts, it can be derived that
		\begin{align*}
			\partial_t \left(\parallel u(\cdot,t) \parallel^2\right)
			&=2 \int_0^1  \partial_t u(x,t) u(x,t) \,\mathrm{d}x \\
			&=2 \int_0^1  (x^{\alpha} u_x)_x(x,t)  u(x,t) \, \mathrm{d}x \\
			&= \left[ 2 x^{\alpha} u_x u\right]_0^1 - 2 \int_0^1 x^{\alpha} u^2(x,t) \, \mathrm{d}x\\
			&=2\left\lbrace -x^{\alpha} u_x u\right\rbrace_{|x=0}- 2 \int_0^1 x^{\alpha} u^2(x,t) \, \mathrm{d}x\\
            &=-2\left\lbrace \dfrac{\beta_0 \beta_1}{\beta_1^2}u^2\right\rbrace_{|x=0}- 2 \int_0^1 x^{\alpha} u^2(x,t) \, \mathrm{d}x.
		\end{align*}
		The assumption stated in Hypothesis \ref{hypo1} ensures that $\dfrac{\beta_0\beta_1}{\beta_1^2} \geq 0$, which completes the proof.
	\end{proof}
Let us first observe that from inequality \eqref{eqq5}, we have
	\begin{equation}\label{eqq6}
		\parallel u(\cdot , t_2) \parallel^{1+M} \leq \exp \dfrac{1}{2} \left( M \max_{x\in\left[0,1\right]} \phi(x,t_1) - (1+M) \min_{x\in\left[0,1\right]} \phi(x,t_2) \right) \parallel u(\cdot,t_1) \parallel^M \parallel u(\cdot,t_3) e^{\frac{\phi}{2}(\cdot, t_3) } \parallel.
	\end{equation}
	Second, we express $\omega$ from \(\parallel u(\cdot,t_3) e^{\frac{\phi}{2}(\cdot, t_3) } \parallel\) as follows
	\begin{align*}
		\parallel u(\cdot,t_3) e^{\frac{\phi}{2}(\cdot, t_3) }  \parallel^2
		&=\int_{\omega}  |u(x,t_3)|^2 \,e^{\phi(x,t_3)} \,dx + \int_{\left(0,1\right)\backslash \omega} |u(x,t_3)|^2 \,e^{\phi(x,t_3)} \,dx \\
		&\leq \exp \left(\max_{x\in\overline{\omega}} \phi (x,t_3) \right) \int_{\omega}  |u(x,t_3)|^2 \,dx + \exp \left(\max_{x\in\overline{(0,1)\backslash \omega}} \phi (x,t_3) \right) \int_0^1  |u(x,t_3)|^2 \,dx.
	\end{align*}
	Combining the above two estimates, \eqref{eqq6} becomes
	\begin{align*}
		\parallel u(\cdot,t_2) \parallel^{1+M}
		&\leq \exp \dfrac{1}{2} \left( M \max_{x\in\left[0,1\right]} \phi(x,t_1) - (1+M) \min_{x\in\left[0,1\right]} \phi(x,t_2) + \max_{x\in\overline{\omega}} \phi(x,t_3) \right) \\
		& \times \parallel u(\cdot,t_1) \parallel^M \parallel u(\cdot,t_3) \parallel_{L^2(\omega)} \\
		&+ \exp \dfrac{1}{2} \left( M \max_{x\in\left[0,1\right]} \phi(x,t_1) - (1+M) \min_{x\in\left[0,1\right]} \phi(x,t_2) + \max_{x\in\overline{\left(0,1\right) \backslash \omega}} \phi(x,t_3) \right) \\
		& \times \parallel u(\cdot,t_1) \parallel^M \parallel u(\cdot,t_3) \parallel.
	\end{align*}
	Now, by applying Lemma \ref{lem3}, we have
    \[\| u(\cdot,t_2)\| \geq \| u(\cdot,T)\| \quad \text{and} \quad \| u(\cdot,t_1)\| \leq \| u(\cdot,0)\|.\]
    This allows us to write
	\begin{align*}
		\parallel u(\cdot,T) \parallel^{1+M}
		&\leq  \exp \dfrac{1}{2} \left( M \max_{x\in\left[0,1\right]} \phi(x,t_1) - (1+M) \min_{x\in\left[0,1\right]} \phi(x,t_2) + \max_{x\in\overline{\omega}} \phi(x,t_3) \right) \\
		& \times \parallel u(\cdot,0) \parallel^M \parallel u(\cdot,t_3) \parallel_{L^2(\omega)} \\
		&+ \exp \dfrac{1}{2} \left( M \max_{x\in\left[0,1\right]} \phi(x,t_1) - (1+M) \min_{x\in\left[0,1\right]} \phi(x,t_2) + \max_{x\in\overline{\left(0,1\right) \backslash \omega}} \phi(x,t_3) \right) \\
		& \times \parallel u(\cdot,0) \parallel^{1+M}.
	\end{align*}
	Subsequently, based on the definition of the weight function $\phi$, it can be concluded that
	\begin{align}\label{eqq7}
		&\parallel u(\cdot,T) \parallel^{1+M} \nonumber\\
		&\leq  \exp \dfrac{s}{2} \left( \dfrac{M}{T-t_1+h} \max_{x\in\left[0,1\right]} \varphi(x) - \dfrac{1+M}{T-t_2+h} \min_{x\in\left[0,1\right]} \varphi(x) + \dfrac{1}{T-t_3+h} \max_{x\in\overline{\omega}} \varphi(x) \right) \nonumber\\
		&\times \parallel u(\cdot,0) \parallel^M \parallel u(\cdot,t_3) \parallel_{L^2(\omega)} \nonumber\\
		&+  \exp \dfrac{s}{2} \left( \dfrac{M}{T-t_1+h} \max_{x\in\left[0,1\right]} \varphi(x) - \dfrac{1+M}{T-t_2+h} \min_{x\in\left[0,1\right]} \varphi(x) + \dfrac{1}{T-t_3+h} \max_{x\in\overline{\left(0,1\right) \backslash \omega}} \varphi(x) \right) \nonumber\\
		& \times  \parallel u(\cdot,0) \parallel^{1+M}.
	\end{align}	

    \noindent\textbf{Step 7.} Choose $t_3=T$, $t_2=T-lh$ and $t_1=T-2lh$ in such a way that $0< 2lh < T$ with $l>1$. With this choice, the expression in \eqref{eqq7} becomes
	\begin{align}\label{ing7}
		\parallel u(\cdot,T) \parallel^{1+M_l} 
		&\leq \exp \dfrac{s}{2 h} \left( \dfrac{M_l}{1+2 l} \max_{x\in\left[0,1\right]} \varphi(x) - \dfrac{1+M_l}{1+l} \min_{x\in\left[0,1\right]} \varphi(x) + \max_{x\in\overline{\omega}} \varphi(x) \right) \nonumber\\
		&\times \parallel u(\cdot,0) \parallel^{M_l} \parallel u(\cdot,T) \parallel_{L^2(\omega)} \\
		&+ \exp \dfrac{s}{2 h} \left( \dfrac{M_l}{1+2l} \max_{x\in\left[0,1\right]} \varphi(x) - \dfrac{1+M_l}{1+l} \min_{x\in\left[0,1\right]} \varphi(x) +  \max_{x\in\overline{\left(0,1\right) \backslash \omega}} \varphi(x) \right) \nonumber\\
		&\times \parallel u(\cdot,0) \parallel^{1+M_l}, \nonumber
	\end{align}
	In this context, $M_l= \dfrac{(l+1)^{C_0}-1}{1-\left(\frac{l+1}{2l+1}\right)^{C_0}}$. Therefore, for $l>1$ sufficiently large one has 
    \[M_l < \frac{(1+l)^{C_0}-1}{1-\left(\frac{2}{3}\right)^{C_0}}.\] Now, observe that $\varphi \leq 0$ on $(0,1)$ which implies that $\displaystyle{\max_{x\in\left[0, 1 \right]}} \varphi(x)=0$. Consequently, one obtains
	\begin{equation*}
		\dfrac{M_l}{1+2l} \max_{x\in\left[0,1\right]} \varphi(x) - \dfrac{1+M_l}{1+l} \min_{x\in\left[0,1\right]} \varphi(x) +  \max_{x\in\overline{\left(0,1\right) \backslash \omega}} \varphi(x)=  - \dfrac{1+M_l}{1+l} \min_{x\in\left[0,1\right]} \varphi(x) +  \max_{x\in\overline{\left(0,1\right) \backslash \omega}} \varphi(x).
	\end{equation*}
	Using the condition \eqref{cond} satisfied by the observation region $\omega$ and knowing that $C_0 \in (0,1)$, we can choose $l>1$ large enough to deduce that
	\begin{equation*}
		\dfrac{M_l}{1+2l} \max_{x\in\left[0,1\right]} \varphi(x) - \dfrac{1+M_l}{1+l} \min_{x\in\left[0,1\right]} \varphi(x) +  \max_{x\in\overline{\left(0,1\right) \backslash \omega}} \varphi(x) < 0,
	\end{equation*}
	and
	\begin{align*}
		\dfrac{M_l}{1+2l} \max_{x\in\left[0,1\right]} \varphi(x) - \dfrac{1+M_l}{1+l} \min_{x\in\left[0,1\right]} \varphi(x) +  \max_{x\in\overline{\omega}} \varphi(x) > 0.
	\end{align*}
	Hence, there exist two positive constants $C_1$ and $C_2$ such that for every $h > 0$ satisfying $0 < 2lh < T$, the following inequality is established
	\begin{align*}
		\parallel u(\cdot,T) \parallel^{1+M_l}
		&\leq  e^{\frac{C_1}{h}} \parallel u(\cdot, 0)\parallel^{M_l} \parallel u(\cdot,T) \parallel_{L^2(\omega)} +  e^{\frac{-C_2}{h}} \parallel u(\cdot,0) \parallel^{1+M_l}.
	\end{align*}
	On the other hand, leveraging Lemma \ref{lem3}, one can notice that $\parallel u(\cdot,T) \parallel \leq  \parallel u(\cdot,0) \parallel $. This combined with the fact that for any $2 lh \geq T$, 
	$1 \leq e^{\frac{2l}{T} C_2} e^{-\frac{C_2}{h}}$ yields that for any $h> 0$,
	\begin{equation*}
		\parallel u(\cdot,T) \parallel^{1+M_l}
		\leq e^{\frac{C_1}{h}} \parallel u(\cdot, 0)\parallel^{M_l} \parallel u(\cdot,T) \parallel_{L^2(\omega)} + e^{\frac{2l}{T} C_2} e^{\frac{-C_2}{h}} \parallel u(\cdot,0) \parallel^{1+M_l}.
	\end{equation*}
	To finalize, we select $h>0$ such that
	\begin{equation*}
		e^{\frac{2l}{T} C_2} e^{\frac{-C_2}{h}} \parallel u(\cdot,0) \parallel^{1+M_l} = \dfrac{1}{2} \parallel u(\cdot,T) \parallel^{1+M_l},
	\end{equation*}
	which leads to
	\begin{equation*}
		e^{\frac{C_1}{h}} = \left(2 e^{ \frac{2l}{T} C_2} \left( \dfrac{\parallel u(\cdot,0) \parallel}{\parallel u(\cdot,T) \parallel} \right)^{1+M_l} \right)^{\frac{C_1}{C_2}}.
	\end{equation*}
	As a consequence, we obtain
	\begin{align*}
		\parallel u(\cdot,T) \parallel^{1+M_l}
		&\leq 2\left[ 2 e^{ \frac{2l}{T} C_2} \left( \dfrac{\parallel u(\cdot,0) \parallel}{\parallel u(\cdot,T) \parallel} \right)^{1+M_l} \right]^{\frac{C_1}{C_2}} \parallel u(\cdot,0) \parallel^{M_l} \parallel u(\cdot,T) \parallel_{L^2(\omega)}.
	\end{align*}
	Then, for $l>1$ one has
	\begin{align*}
		\parallel u(\cdot,T) \parallel^{1+M_l+\frac{C_1}{C_2}(1+M_l)}
		\leq &e^{\left(1+\frac{C_1}{C_2}\right)\ln(2)+ \frac{2l}{T} C_1} \parallel u(\cdot,0) \parallel^{M_l + (1+M_l)\frac{C_1}{C_2}} \parallel u(\cdot,T) \parallel_{L^2(\omega)}\\
        \leq &e^{2l\left(1+\frac{C_1}{C_2}\right) \left(1+\frac{C_1}{T}\right)}  \parallel u(\cdot,0) \parallel^{M_l + (1+M_l)\frac{C_1}{C_2}} \parallel u(\cdot,T) \parallel_{L^2(\omega)}.
	\end{align*}
	Setting $\rho=\dfrac{1}{1+M_l + (1+M_l)\frac{C_1}{C_2}}$, $\mathcal{C}=2l \left(1+\dfrac{C_1}{C_2} \right)  \max\left(C_1, 1 \right)$ allows us to complete the proof.
\end{proof}

\section{Null approximate impulse controllability}\label{section4}
	We are now prepared to prove our principal result concerning the null approximate impulse controllability of the impulsive system \eqref{1.1}. To begin the proof, we present the following lemma, which is a direct consequence of the observation estimate derived in Lemma \ref{lem1}.
	
	\begin{lemma}\label{lem6}
    Assume that Hypothesis \ref{hypo1} holds true and the observation region $\omega$ meets the geometric condition \ref{cond}. Then, one can find positive constants $\mathcal{C}$ and $\beta$ such that the estimate remains valid for all $\epsilon>0$,
		\begin{equation}\label{eqq13}
			\parallel u(\cdot,T) \parallel^2 \leq \left( \dfrac{e^{\mathcal{C} \left(1+\frac{1}{T}\right)} }{\epsilon^{\beta}} \right)^2  \parallel u(\cdot,T) \parallel^2_{L^2(\omega)} + \epsilon^2 \parallel u(\cdot,0) \parallel^2.
		\end{equation}
	\end{lemma}
	
	\begin{proof}
		According to Lemma \ref{lem1}, there exist a positive constant $\mathcal{C}$ and a parameter $\rho \in (0,1)$ such that
		\begin{equation*}
			\parallel u(\cdot,T) \parallel^2 \leq \left(  e^{ \mathcal{C} \left(1+\frac{1}{T}\right)} \parallel u(\cdot,T) \parallel_{L^2(\omega)}\right)^{2 \rho} \parallel u(\cdot,0) \parallel^{2(1-\rho)}.
		\end{equation*}
		Let $\epsilon>0$. We write
		\begin{align*}
			&\left( e^{ \mathcal{C} \left(1+\frac{1}{T}\right)}  \parallel u(\cdot,T) \parallel_{L^2(\omega)}\right)^{2 \rho} \parallel u(\cdot,0) \parallel^{2(1-\rho)}\\
			&=\left( \dfrac{\left(1-\rho \right)^{\frac{1-\rho}{2\rho}}}{\epsilon^{\frac{1-\rho}{\rho}}} e^{ \mathcal{C} \left(1+\frac{1}{T}\right)} \parallel u(\cdot,T) \parallel_{L^2(\omega)}   \right)^{2 \rho} \times \left( \epsilon\left(\dfrac{1}{1-\rho} \right)^{\frac{1}{2}} \parallel u(\cdot,0) \parallel  \right)^{2(1-\rho)}.
		\end{align*}
		An application of Young’s inequality yields the following estimate, valid for all $\epsilon > 0$.
		\begin{align*}
			\parallel u(\cdot,T) \parallel^2 
			\leq &\rho \left[ e^{ \mathcal{C} \left(1+\frac{1}{T}\right)} \left(\dfrac{1-\rho}{\epsilon^2} \right)^{\frac{1-\rho}{2 \rho}} \right]^2 \parallel u(\cdot,T) \parallel_{L^2(\omega)}^2 + \epsilon^2 \parallel u(\cdot,0) \parallel^2\\
			\leq & \dfrac{1}{\epsilon^{\frac{2(1-\rho)}{\rho}}}e^{ 2\mathcal{C} \left(1+\frac{1}{T}\right)} \parallel u(\cdot,T) \parallel_{L^2(\omega)}^2 + \epsilon^2 \parallel u(\cdot,0) \parallel^2.
		\end{align*}
		This computation confirms the inequality \eqref{eqq13}, where the parameter $\beta$  is given explicitly by
		\[\beta=\dfrac{1-\rho}{\rho}.\]
	\end{proof}
    \begin{remark}
    The analyticity of the semigroup \((e^{tA})_{t \geq 0}\) implies that the range \(\mathcal{R}(e^{TA})\) is dense in \(L^2(0,1)\). Consequently, null approximate impulse controllability is equivalent to approximate impulse controllability, meaning that  
    \[\forall y_0, y_T \in L^2(0,1), \; \forall \epsilon > 0, \; \exists h(\cdot, \tau) \in L^2(\omega) \;\; \text{such that} \;\; \|y(\cdot, T) - y_T\| \leq \epsilon \|y_0\|.\]  
    In light of this equivalence, we shall henceforth refer to this property simply as *approximate controllability*.  
    \end{remark}
	The main impulse controllability result of this work reads as follows.
	\begin{theorem}\label{thm1} Let us assume that Hypothesis \ref{hypo1} holds and that the control region $\omega$ satisfies the geometric condition \eqref{cond}. Under these conditions, the system \eqref{1.1} is approximately impulse controllable for all $T > 0$. In addition, for any $\epsilon > 0$, the cost of the null approximate impulse control function at time $T$ satisfies the following inequality:
    \[\mathcal{P}(T, \epsilon) \leq \frac{1}{\epsilon^{\beta}} e^{ \mathcal{C} \left(1 + \frac{1}{T - \tau}\right)},\]
    where the positive constants $\mathcal{C}$ and $\beta$ are obtained from the estimate \eqref{eqq13}.
	\end{theorem}
	
	\begin{proof}
		Let $\epsilon > 0$, $y_0 \in L^2(0,1)$, and define $K = \frac{1}{\epsilon^{\beta}} e^{ \mathcal{C} \left(1 + \frac{1}{T - \tau}\right)}$. We introduce the functional $\mathcal{J}_{\epsilon, K}: L^2(0,1) \to \mathbb{R}$ defined by
        \[\mathcal{J}_{\epsilon, K}(v_0) = \frac{K^2}{2} \| v(\cdot, T - \tau) \|^2_{L^2(\omega)} + \frac{\epsilon^2}{2} \| v_0 \|^2 + \langle y_0, v(\cdot, T) \rangle,\]
        where $v$ is the solution of the non-impulsive system \eqref{1.2} corresponding to the initial condition $v_0$. By applying the estimate \eqref{eqq13} and utilizing the fact that
        \[\|v(\cdot, T)\|_{L^2(\omega)} \leq \|v(\cdot, T - \tau)\|_{L^2(\omega)},\]
        we derive the following inequality:
        \[\frac{K^2}{2} \|v(\cdot, T - \tau)\|^2_{L^2(\omega)} \geq \frac{1}{4} \left( \frac{e^{\mathcal{C} \left(1 + \frac{1}{T}\right)}}{\epsilon^{\beta}} \right)^2 \|v(\cdot, T)\|^2_{L^2(\omega)} \geq \frac{1}{4} \|v(\cdot, T)\|^2 - \frac{\epsilon^2}{4} \|v_0\|^2.\]
        Thus, by applying Young's inequality, we derive the following result
        \begin{align*}
            \mathcal{J}_{\epsilon,K}(v_0) 
            \geq &\dfrac{1}{4} \|v(\cdot,T)\|^2 + \dfrac{\epsilon^2}{4} \|v_0\|^2 + \left\langle y_0 , v(\cdot,T) \right\rangle\\
            \geq &\dfrac{\epsilon^2}{4} \|v_0\|^2 - \|y_0\|^2.
        \end{align*}
        This leads to the coercivity of the functional $\mathcal{J}_{\epsilon, K}$. Moreover, owing to its strict convexity and $\mathcal{C}^1$ differentiability, we can deduce that $\mathcal{J}_{\epsilon, K}$ attains a unique minimizer $\tilde{v}_0 \in L^2(0,1)$, which is characterized by
		\begin{equation}\label{eqq19}
			\mathcal{J}_{\epsilon,K}(\tilde{v}_0)= \min_{v_0 \in L^2(0,1)} \mathcal{J}_{\epsilon,K}(v_0).
		\end{equation}
		Additionally, $\mathcal{J}_{\epsilon, K}'(\tilde{v}_0) z_0 = 0$ for every $z_0 \in L^2(0,1)$, which results in the following bound for any $z_0$
		\begin{equation}\label{eqq1}
			K^2 \int_{\omega}  \tilde{v}(x,T-\tau) z(x,T-\tau) \,\mathrm{d}x + \epsilon^2 \int_{0}^{1}  \tilde{v}_0(x) z_0(x) \,\mathrm{d}x + \int_{0}^{1}  y_0(x) z(x,T) \,\mathrm{d}x =0.
		\end{equation}
		Here, $\tilde{v}$ and $z$ represent the solutions of the non-impulsive system \eqref{1.2}, associated with the initial conditions $\tilde{v}_0$ and $z_0$, respectively. To proceed, we multiply the first equation in \eqref{1.1} by $z(\cdot, T - t)$ for all $t \in (0, T) \setminus \{\tau\}$, and integrate over the spatial domain $(0, 1)$, leading to the following result
		\begin{equation*}
			\int_{0}^{1} \Big(\partial_t y - (x^{\alpha} y_x)_x\Big)(x,t) z(x,T-t) \,\mathrm{d}x=0.
		\end{equation*}
		By applying integration by parts, one gets
		\begin{equation}\label{star}
			\partial_t	\left(\int_{0}^{1}  y(x,t) z(x,T-t)\,\mathrm{d}x\right)=0.
		\end{equation}
		Next, we integrate \eqref{star} over $(0,\tau)$, obtaining
		\begin{equation}\label{starr}
			\int_{0}^{1}  y(x,\tau^{-}) z(x,T-\tau)\,\mathrm{d}x=\int_{0}^{1}  y_0(x) z(x,T) \,\mathrm{d}x.
		\end{equation}
		Integrating again \eqref{star} over $(\tau,T)$, gives
		\begin{equation}\label{starrr}
			\int_{0}^{1}  y(x,\tau) z(x,T-\tau)\,\mathrm{d}x= \int_{0}^{1}  y(x,T) z_0(x) \,\mathrm{d}x.
		\end{equation}
		By associating \eqref{starr} and \eqref{starrr}, and taking into account that $y(\cdot,\tau)=y(\cdot,\tau^{-})+ 1_{\omega}h(\cdot,\tau)$, we infer that
		\begin{equation}\label{eqq2}
			\int_{\omega}  f(x,\tau) z(x,T-\tau)\,\mathrm{d}x - \int_{0}^{1}  y(x,T)z_0(x)\,\mathrm{d}x + \int_{0}^{1}  y_0(x) z(x,T)\,\mathrm{d}x=0.
		\end{equation}
		Based on \eqref{eqq1} and \eqref{eqq2}, selecting $f(\cdot,\tau)=K^2 \tilde{v}(\cdot,T-\tau)$ yields
		\begin{equation}\label{eqq3}
			\int_{0}^{1}  \Big(y(x,T)+ \epsilon^2 \tilde{v}_0(x)\Big)z_0(x)\,\mathrm{d}x=0,\quad \forall z_0\in L^2(0,1).
		\end{equation}
		This leads to the conclusion that $y(\cdot,T)=-\epsilon^2 \tilde{v}_0$. Additionally, by setting $z_0=\tilde{v}_0$, and applying both \eqref{eqq1} and Young's inequality, we can derive
		\begin{equation}\label{eqq4}
			K^2 \parallel \tilde{v}(\cdot,T-\tau) \parallel^2_{L^2(\omega)}+ \epsilon^2 \parallel \tilde{v}_0 \parallel^2 \leq \dfrac{1}{2}\parallel y_0 \parallel^2 + \dfrac{1}{2} \parallel \tilde{v}(\cdot,T) \parallel^2.
		\end{equation}
		Now, by using \eqref{eqq13}  and applying Lemma \ref{lem3}, we can now make the following estimation
		\begin{align}\label{eqq8}
			\parallel \tilde{v}(\cdot,T) \parallel^2 
			&\leq \left( \dfrac{1}{\epsilon^{\beta}} e^{\mathcal{C}\left(1+\frac{1}{T}\right)}  \right)^2 \parallel \tilde{v}(\cdot,T) \parallel^2_{L^2(\omega)}+\epsilon^2 \parallel \tilde{v}_0 \parallel^2 \nonumber\\
			&\leq K^2 \parallel \tilde{v}(\cdot,T-\tau) \parallel^2_{L^2(\omega)}+\epsilon^2 \parallel \tilde{v}_0 \parallel^2.
		\end{align}
		Finally, combining \eqref{eqq4} and \eqref{eqq8} , it follows that
		\begin{equation*}
			K^2 \parallel \tilde{v}(T-\tau) \parallel^2_{L^2(\omega)}+ \epsilon^2 \parallel \tilde{v}_0 \parallel^2 \leq \parallel y_0 \parallel^2.
		\end{equation*}
		Remember that $y(\cdot,T)=-\epsilon^2 \tilde{v}_0$ and $f(\cdot,\tau)=K^2 \tilde{v}(\cdot,T-\tau)$, it follows that
		\begin{equation}\label{eqq9}
			\dfrac{1}{K^2} \parallel f(\cdot,\tau) \parallel^2_{L^2(\omega)}+ \dfrac{1}{\epsilon^2} \parallel y(\cdot,T) \parallel^2 \leq \parallel y_0 \parallel^2.
		\end{equation}
		This concludes the proof of Theorem \ref{thm1}.
	\end{proof}

    \section{Conclusion}\label{section5}
    This paper investigates an impulse control problem for a one-dimensional degenerate heat equation in divergence form, subject to Robin boundary conditions. When the control is supported within a prescribed subinterval of the spatial domain (as specified in condition \eqref{cond}), we establish null approximate controllability and derive an explicit upper bound on the cost of the impulse control. The analysis relies on a crucial observability estimate at a single time instant for solutions to the associated non-impulsive system. However, if the control support is strictly confined to the interior of the physical domain $(0,1)$, the problem remains unresolved, even for the degenerate equation without classical Dirichlet or Neumann boundary conditions.\\

\textbf{Author Contributions } H.E.B., and I.O. participated equally in providing the mathematical arguments and contributed to the writing and review of the manuscript.\\

\textbf{Funding} The authors declare that this research did not receive any funding.\\

\textbf{Data Availability} No datasets were generated or analysed during the current study.

\section*{Declarations}

\textbf{Conflict of interest} The authors have no conflicts of interest to declare that are relevant to the content of this
article.\\

\textbf{Ethics approval} The authors confirm that this work did not involve any human and/or animal studies and, as a result, did not require ethics approval.\\

\textbf{Competing interests} The authors declare no competing interests.


\end{document}